\documentclass{article}
\newtheorem{theorem}{Theorem}

\newtheorem{corollary}[theorem]{Corollary}

\newtheorem{lemma}[theorem]{Lemma}

\newtheorem{remark}[theorem]{Remark}

\newenvironment{proof}[1][Proof]{\noindent\textbf{#1.} }{\ \rule{0.5em}{0.5em}}
\input{tcilatex}
\begin{document}

\title{SERIES WITH HERMITE POLYNOMIALS AND APPLICATIONS}
\author{Khristo N. Boyadzhiev \\
%EndAName
Department of Mathematics and Statistics,\\
Ohio Northern University Ada, \ Ohio 45810, USA,\\
k-boyadzhiev@onu.edu \and Ayhan Dil \\
%EndAName
Department of Mathematics,\\
Akdeniz University, 07058 Antalya Turkey\\
adil@akdeniz.edu.tr}
\maketitle

\begin{abstract}
We obtain a series transformation formula involving the classical Hermite
polynomials. We then provide a number of applications using appropriate
binomial transformations. Several of the new series involve Hermite
polynomials and harmonic numbers, Lucas sequences, exponential and geometric
numbers. We also obtain a series involving both Hermite and Laguerre
polynomials, and a series with Hermite polynomials and Stirling numbers of
the second kind.

\textbf{2000 Mathematics Subject Classification. }11B83, 33C45, 05A19

\textbf{Key words: }Hermite polynomials, Laguerre polynomials, harmonic
numbers, Fibonacci numbers, exponential numbers, geometric numbers, Stirling
numbers of the second kind, binomial transformation, Euler series
transformation, generating functions.
\end{abstract}

\section{Introduction and main result}

Let $H_{n}\left( x\right) $\ be the Hermite polynomials defined by the
generating function%
\begin{equation}
e^{2xt-t^{2}}=\sum_{n=0}^{\infty }H_{n}\left( x\right) \frac{t^{n}}{n!}
\label{1}
\end{equation}%
and satisfying the Rodrigues formula%
\begin{equation}
H_{n}\left( x\right) =\left( -1\right) ^{n}e^{x^{2}}\left( \frac{d}{dx}%
\right) ^{n}e^{-x^{2}}.  \label{2}
\end{equation}%
In this paper we present and discuss a series transformation formula for the
series%
\[
\sum_{n=0}^{\infty }a_{n}H_{n}\left( x\right) \frac{t^{n}}{n!} 
\]%
where%
\begin{equation}
f\left( t\right) =\sum_{k=0}^{\infty }a_{k}t^{k}  \label{f}
\end{equation}%
is an arbitrary function, analytical in a neighborhood of zero. Most of our
results are based on the theorem:

\begin{theorem}
\label{MT1}With $f\left( t\right) $\ as above the following series
transformation formula holds%
\begin{equation}
\sum_{n=0}^{\infty }a_{n}H_{n}\left( x\right) \frac{t^{n}}{n!}%
=e^{2xt-t^{2}}\sum_{n=0}^{\infty }\left( -1\right) ^{n}H_{n}\left(
x-t\right) \frac{t^{n}}{n!}\left\{ \sum_{k=0}^{n}\binom{n}{k}\left(
-1\right) ^{k}a_{k}\right\} .  \label{3}
\end{equation}
\end{theorem}

For example, when $a_{k}=1$ for all $k=0,1,\ldots $, then $\sum_{k=0}^{n}%
\binom{n}{k}\left( -1\right) ^{k}=\left( 1-1\right) ^{n}=0$ except when $n=0$%
, and $\left( \ref{3}\right) $ turns into $\left( \ref{1}\right) $.

For the proof of the theorem we need the following lemma.

\begin{lemma}
\label{LM}Let%
\[
g\left( t\right) =\sum_{k=0}^{\infty }b_{k}t^{k} 
\]%
be another analytical function like $f\left( t\right) $. Then%
\begin{equation}
\sum_{n=0}^{\infty }a_{n}b_{n}t^{n}=\sum_{n=0}^{\infty }\frac{\left(
-1\right) ^{n}g^{\left( n\right) }\left( t\right) }{n!}t^{n}\left\{
\sum_{k=0}^{n}\binom{n}{k}\left( -1\right) ^{k}a_{k}\right\} .  \label{4}
\end{equation}%
in some neighborhood of zero where both sides are convergent.
\end{lemma}

The proof of the lemma can be found in \cite{B5}. This result originates in
the works of Euler. In a modified form $\left( \ref{4}\right) $ appears in [%
\cite{Ri}, Chapter 6, problem 19, p. 245].

Now we are ready to give the proof of the Theorem \ref{MT1}.

\begin{proof}
We use $\left( \ref{4}\right) $ with $g\left( t\right) =e^{2xt-t^{2}}$. From
Rodrigues' formula it follows that 
\begin{eqnarray*}
\left( \frac{d}{dt}\right) ^{n}e^{2xt-t^{2}} &=&\left( \frac{d}{dt}\right)
^{n}e^{x^{2}}e^{-\left( x-t\right) ^{2}}=e^{x^{2}}\left( -1\right)
^{n}\left( \frac{d}{dx}\right) ^{n}e^{-\left( x-t\right) ^{2}} \\
&=&e^{x^{2}}\left( -1\right) ^{n}\left( \frac{d}{d\left( x-t\right) }\right)
^{n}e^{-\left( x-t\right) ^{2}}=e^{2xt-t^{2}}H_{n}\left( x-t\right) .
\end{eqnarray*}%
That is,%
\begin{equation}
\left( \frac{d}{dt}\right) ^{n}e^{2xt-t^{2}}=e^{2xt-t^{2}}H_{n}\left(
x-t\right)  \label{5}
\end{equation}%
from which $\left( \ref{3}\right) $ follows in view of Lemma \ref{LM}. The
proof is completed.
\end{proof}

In the next section we present several corollaries resulting from the
theorem. For these corollaries we use appropriate binomial transforms \cite%
{H1, Ri}.%
\begin{equation}
b_{n}=\sum_{k=0}^{n}\binom{n}{k}\left( -1\right) ^{k}a_{k},\text{ \ \ }%
n=0,1,\ldots  \label{6}
\end{equation}%
We also note that the binomial transform $\left( \ref{6}\right) $ can be
inverted, with inversion sequence%
\begin{equation}
a_{n}=\sum_{k=0}^{n}\binom{n}{k}\left( -1\right) ^{k}b_{k},\text{ \ \ }%
n=0,1,\ldots  \label{9}
\end{equation}

It is good to notice that such binomial transforms can be computed
conveniently by using the Euler series transformation formula (see \cite{B2,
KK, P}).%
\[
\frac{1}{1-\lambda t}f\left( \frac{\mu t}{1-\lambda t}\right)
=\sum_{n=0}^{\infty }t^{n}\left\{ \sum_{k=0}^{n}\binom{n}{k}\mu ^{k}\lambda
^{n-k}a_{k}\right\} 
\]%
where $f\left( n\right) $\ is as in $\left( \ref{f}\right) $ and $\lambda $, 
$\mu $ are parameters. With $\lambda =1$ and $\mu =1$, $\mu =-1$ we have
correspondingly%
\begin{equation}
\frac{1}{1-t}f\left( \frac{t}{1-t}\right) =\sum_{n=0}^{\infty }t^{n}\left\{
\sum_{k=0}^{n}\binom{n}{k}a_{k}\right\} ,  \label{7}
\end{equation}%
\begin{equation}
\frac{1}{1-t}f\left( \frac{-t}{1-t}\right) =\sum_{n=0}^{\infty }t^{n}\left\{
\sum_{k=0}^{n}\binom{n}{k}\left( -1\right) ^{k}a_{k}\right\} .  \label{8}
\end{equation}%
In the sequel we shall use $\left( \ref{7}\right) $ or $\left( \ref{8}%
\right) $.

Lastly, we also introduce two transformation formulas which we need later.

Suppose we are given an entire function $f$\ and a function $g$, analytic in
a region containing the disk $K=\{z:|z|<R\}$. Hence we have the following
transformation formula $\left( \cite{B3}\right) $,%
\begin{equation}
\sum_{n=0}^{\infty }\frac{g^{\left( n\right) }\left( 0\right) }{n!}f\left(
n\right) x^{n}=\sum_{n=0}^{\infty }\frac{f^{\left( n\right) }\left( 0\right) 
}{n!}\sum_{k=0}^{n}\QATOPD\{ \} {n}{k}x^{k}g^{\left( k\right) }\left(
x\right)   \label{10}
\end{equation}%
where $\QATOPD\{ \} {n}{k}$ are the the Stirling numbers of the second kind.

A generalization of $\left( \ref{10}\right) $ is given in $\cite{D4}$ as:

\begin{equation}
\sum_{n=r}^{\infty }\frac{g^{\left( n\right) }\left( 0\right) }{n!}\binom{n}{%
r}\frac{r!}{n^{r}}f_{r}\left( n\right) t^{n}=\sum_{n=r}^{\infty }\frac{%
f^{\left( n\right) }\left( 0\right) }{n!}\sum_{k=0}^{n}\QATOPD\{ \}
{n}{k}_{r}t^{k}g^{\left( k\right) }\left( t\right)  \label{11}
\end{equation}%
where $f_{r}\left( x\right) $ denotes the Maclaurin series of $f\left(
x\right) $ exclude the first $r$ terms and $\QATOPD\{ \} {n}{k}_{r}$ are the
the $r$-Stirling numbers of the second kind (\cite{BR}).

These two transformation formulas tightly coupled with the derivative
operator $\left( t\frac{d}{dt}\right) $ which defined as:%
\begin{equation}
\left( t\frac{d}{dt}\right) f\left( t\right) =tf^{\prime }\left( t\right) .
\label{d}
\end{equation}%
From $\left( \ref{d}\right) $ it is easy to see that%
\begin{equation}
\left( t\frac{d}{dt}\right) ^{n}f\left( t\right) =\sum_{k=0}^{n}\QATOPD\{ \}
{n}{k}t^{k}f^{\left( k\right) }\left( t\right) .  \label{dd}
\end{equation}

\section{Applications}

This section consists of two parts. The first part depends on the Theorem %
\ref{MT1} and transformation formulas $\left( \ref{7}\right) $, $\left( \ref%
{8}\right) $. Second part depends on transformation formulas $\left( \ref{10}%
\right) $ and $\left( \ref{11}\right) $.

In all series below the variable is taken in some neighborhood of zero,
small enough to ensure convergence of both sides.

\subsection{Transformation formulas $\left( \protect\ref{7}\right) $\ and $%
\left( \protect\ref{8}\right) $}

Our first application is the following.

\begin{corollary}
We have%
\begin{equation}
\sum_{n=0}^{\infty }H_{n}\left( x\right) \frac{t^{n}}{\left( n+1\right) !}%
=e^{2xt-t^{2}}\sum_{n=0}^{\infty }\left( -1\right) ^{n}H_{n}\left(
x-t\right) \frac{t^{n}}{\left( n+1\right) !}.  \label{12}
\end{equation}
\end{corollary}

\begin{proof}
Take $a_{k}=\frac{1}{k+1}$ in $\left( \ref{3}\right) $ and use the fact that%
\begin{equation}
\frac{1}{n+1}=\sum_{k=0}^{n}\binom{n}{k}\left( -1\right) ^{k}\frac{1}{k+1}.
\label{13}
\end{equation}%
In this case%
\begin{equation}
f\left( t\right) =\frac{-\ln \left( 1-t\right) }{t}  \label{14}
\end{equation}%
which is invariant under the transformation%
\begin{equation}
f\left( t\right) \rightarrow \frac{1}{1-t}f\left( \frac{-t}{1-t}\right)
\label{15}
\end{equation}%
and $\left( \ref{13}\right) $ follows from $\left( \ref{8}\right) $.
\end{proof}

For the rest of the section we shall use subtitles for the convenience of
the reader.

\subsubsection{Harmonic numbers}

\begin{remark}
Since the notation $H_{n}\left( x\right) $ for Hermite polynomials and the
standard notation $H_{n}$ for harmonic numbers are very much alike, in order
to avoid confusion we shal use here the notation $h_{n}$ for harmonic
numbers.
\end{remark}

Now let%
\begin{equation}
h_{n}=1+\frac{1}{2}+\cdots +\frac{1}{n}\text{, }h_{0}=0  \label{16}
\end{equation}%
be the usual harmonic numbers. They have the representation as binomial
transform%
\begin{equation}
-h_{n}=\sum_{k=1}^{n}\binom{n}{k}\left( -1\right) ^{k}\frac{1}{k}.
\label{17}
\end{equation}%
Formula $\left( \ref{17}\right) $ corresponds to $f\left( t\right) =-\ln
\left( 1-t\right) $ and to the generating function%
\begin{equation}
\frac{1}{1-t}f\left( \frac{-t}{1-t}\right) =\frac{\ln \left( 1-t\right) }{1-t%
}=-\sum_{n=0}^{\infty }h_{n}t^{n}.  \label{18}
\end{equation}%
Therefore, we obtain the following corollary (with $a_{k}=\frac{1}{k}$).

\begin{corollary}
\begin{equation}
\sum_{n=1}^{\infty }H_{n}\left( x\right) \frac{t^{n}}{n!n}%
=e^{2xt-t^{2}}\sum_{n=1}^{\infty }\left( -1\right) ^{n-1}H_{n}\left(
x-t\right) h_{n}\frac{t^{n}}{n!}.  \label{19}
\end{equation}
\end{corollary}

Next we have an identity \textquotedblleft symmetric\textquotedblright\ to $%
\left( \ref{19}\right) $.

\begin{corollary}
\begin{equation}
\sum_{n=1}^{\infty }H_{n}\left( x\right) h_{n}\frac{t^{n}}{n!}%
=e^{2xt-t^{2}}\sum_{n=1}^{\infty }\left( -1\right) ^{n-1}H_{n}\left(
x-t\right) \frac{t^{n}}{n!n}  \label{20}
\end{equation}
\end{corollary}

\begin{proof}
This series results from $\left( \ref{3}\right) $ and the inversion of $%
\left( \ref{17}\right) $%
\begin{equation}
-\frac{1}{n}=\sum_{k=1}^{n}\binom{n}{k}\left( -1\right) ^{k}h_{k}.
\label{21}
\end{equation}
\end{proof}

\begin{corollary}
\begin{equation}
\sum_{n=0}^{\infty }H_{n}\left( x\right) h_{n}\frac{t^{n}}{\left( n+1\right)
!}=e^{2xt-t^{2}}\sum_{n=0}^{\infty }\left( -1\right) ^{n}H_{n}\left(
x-t\right) \frac{t^{n}}{\left( n+1\right) !}.  \label{22}
\end{equation}
\end{corollary}

\begin{proof}
The representation follows from $\left( \ref{3}\right) $ and the binomial
identity,%
\begin{equation}
\sum_{k=0}^{n}\binom{n}{k}\left( -1\right) ^{k}\frac{h_{k}}{k+1}=\frac{-h_{n}%
}{n+1}.  \label{23}
\end{equation}%
The generating function here is%
\begin{equation}
f\left( t\right) =\frac{1}{2t}\ln ^{2}\left( 1-t\right) =\sum_{k=0}^{\infty }%
\frac{h_{k}}{k+1}t^{k}  \label{24}
\end{equation}%
with $a_{k}=\frac{h_{k}}{k+1}$. This function has the property%
\begin{equation}
-f\left( t\right) =\frac{1}{1-t}f\left( \frac{-t}{1-t}\right)  \label{25}
\end{equation}%
and we apply $\left( \ref{8}\right) $ to obtain $\left( \ref{23}\right) $.
\end{proof}

The next corollary uses the \textquotedblleft square\textquotedblright\
harmonic numbers%
\begin{equation}
h_{n}^{\left( 2\right) }=1+\frac{1}{2^{2}}+\cdots +\frac{1}{n^{2}}.
\label{26}
\end{equation}

\begin{lemma}
\label{L2}The numbers $h_{n}^{\left( 2\right) }$\ have the binomial
representation%
\begin{equation}
-h_{n}^{\left( 2\right) }=\sum_{k=1}^{n}\binom{n}{k}\left( -1\right) ^{k}%
\frac{h_{k}}{k}.  \label{27}
\end{equation}
\end{lemma}

\begin{proof}
It is easy to see that the generating function for the sequence $a_{k}=\frac{%
h_{k}}{k}$\ is%
\begin{equation}
f\left( t\right) =Li_{2}\left( t\right) +\frac{1}{2}\ln ^{2}\left( 1-t\right)
\label{28}
\end{equation}%
where%
\begin{equation}
Li_{2}\left( t\right) =\sum_{n=1}^{\infty }\frac{t^{n}}{n^{2}}  \label{29}
\end{equation}%
is the dilogarithm function. This verification can be done by
differentiating in $\left( \ref{28}\right) $ and using $\left( \ref{19}%
\right) $. The dilogarithm satisfies the Landen identity%
\[
Li_{2}\left( t\right) +\frac{1}{2}\ln ^{2}\left( 1-t\right) =-Li_{2}\left( 
\frac{-t}{1-t}\right) , 
\]%
that is, $f\left( t\right) =-Li_{2}\left( \frac{-t}{1-t}\right) $, and
therefore%
\begin{equation}
\frac{1}{1-t}f\left( \frac{-t}{1-t}\right) =\frac{-1}{1-t}Li_{2}\left(
t\right) =-\sum_{n=1}^{\infty }h_{n}^{\left( 2\right) }t^{n}  \label{30}
\end{equation}%
so $\left( \ref{27}\right) $ follows from $\left( \ref{8}\right) $. The
proof is completed.
\end{proof}

Lemma \ref{L2}\ yields the following.

\begin{equation}
\sum_{n=1}^{\infty }H_{n}\left( x\right) h_{n}\frac{t^{n}}{n!n}%
=e^{2xt-t^{2}}\sum_{n=1}^{\infty }\left( -1\right) ^{n-1}H_{n}\left(
x-t\right) h_{n}^{\left( 2\right) }\frac{t^{n}}{n!}.  \label{31}
\end{equation}%
The symmetric version of $\left( \ref{31}\right) $ based on the inversion of 
$\left( \ref{27}\right) $ is left to the reader.

\subsubsection{Bilinear series with Hermite and Laguerre polynomials}

An interesting and curious two-variable series exists involving the Hermite
polynomials together with the Laguerre polynomials, $L_{n}\left( x\right) $, 
$n=0,1,\ldots $. The Laguerre polynomials have the representation%
\begin{equation}
L_{n}\left( z\right) =\sum_{k=0}^{n}\binom{n}{k}\left( -1\right) ^{k}\frac{%
z^{k}}{k!}  \label{32}
\end{equation}%
which follows immediately by comparing their generating function%
\begin{equation}
\frac{1}{1-t}\exp \left( \frac{-zt}{1-t}\right) =\sum_{n=0}^{\infty
}L_{n}\left( z\right) t^{n}  \label{33}
\end{equation}%
to $\left( \ref{8}\right) $ with $f\left( t\right) =e^{zt}$. From $\left( %
\ref{32}\right) $ we obtain via $\left( \ref{3}\right) $.

\begin{corollary}
\begin{equation}
\sum_{n=0}^{\infty }H_{n}\left( x\right) \frac{\left( zt\right) ^{n}}{\left(
n!\right) ^{2}}=e^{2xt-t^{2}}\sum_{n=0}^{\infty }\left( -1\right)
^{n}H_{n}\left( x-t\right) L_{n}\left( z\right) \frac{t^{n}}{n!}.  \label{34}
\end{equation}
\end{corollary}

Even more interesting is the \textquotedblleft symmetric\textquotedblright\
series resulting from the inversion of $\left( \ref{32}\right) $.%
\[
\frac{z^{n}}{n!}=\sum_{k=0}^{n}\binom{n}{k}\left( -1\right) ^{k}L_{k}\left(
z\right) . 
\]%
Namely, we have the following.

\begin{corollary}
\begin{equation}
\sum_{n=0}^{\infty }H_{n}\left( x\right) L_{n}\left( z\right) \frac{t^{n}}{n!%
}=e^{2xt-t^{2}}\sum_{n=0}^{\infty }\left( -1\right) ^{n}H_{n}\left(
x-t\right) \frac{\left( zt\right) ^{n}}{\left( n!\right) ^{2}}.  \label{35}
\end{equation}
\end{corollary}

\subsubsection{Bilinear series with Hermite polynomials}

The next result includes pairs of Hermite polynomials. We start with a
well-known property%
\begin{equation}
H_{n}\left( z+y\right) =\sum_{k=0}^{n}\binom{n}{k}\left( 2y\right)
^{n-k}H_{k}\left( z\right) .  \label{36}
\end{equation}%
This identity can be written as a binomial transform (dividing both sides by 
$\left( 2y\right) ^{n}$\ and replacing by $y$ by $-y$),%
\begin{equation}
\left( -1\right) ^{n}\frac{H_{n}\left( z-y\right) }{\left( 2y\right) ^{n}}%
=\sum_{k=0}^{n}\binom{n}{k}\left( -1\right) ^{k}\frac{H_{k}\left( z\right) }{%
\left( 2y\right) ^{k}}.  \label{37}
\end{equation}%
Applying now $\left( \ref{3}\right) $ with $a_{k}=\frac{H_{k}\left( z\right) 
}{\left( 2y\right) ^{k}}$ we obtain,%
\begin{equation}
\sum_{n=0}^{\infty }H_{n}\left( x\right) H_{n}\left( z\right) \frac{1}{n!}%
\frac{t^{n}}{\left( 2y\right) ^{n}}=e^{2xt-t^{2}}\sum_{n=0}^{\infty
}H_{n}\left( x-t\right) H_{n}\left( z-y\right) \frac{1}{n!}\frac{t^{n}}{%
\left( 2y\right) ^{n}}.  \label{38}
\end{equation}%
Replacing here $t$\ by $2ty$\ we obtain our next result.

\begin{corollary}
For all $x,y,z$ and for $t$ small enough we have%
\begin{equation}
\sum_{n=0}^{\infty }H_{n}\left( x\right) H_{n}\left( z\right) \frac{t^{n}}{n!%
}=e^{4xyt-4y^{2}t^{2}}\sum_{n=0}^{\infty }H_{n}\left( x-2yt\right)
H_{n}\left( z-y\right) \frac{t^{n}}{n!}.  \label{39}
\end{equation}
\end{corollary}

Note that the variable $y$\ does not appear in the left-hand side.

We can compare this expansion to the well-known bilinear series [\cite{P},
p.198], [\cite{T}, p.167],%
\begin{equation}
\sum_{n=0}^{\infty }H_{n}\left( x\right) H_{n}\left( z\right) \frac{t^{n}}{n!%
}=\frac{1}{\sqrt{1-4t^{2}}}\exp \left\{ x^{2}-\frac{\left( x-2zt\right) ^{2}%
}{1-4t^{2}}\right\} ,  \label{40}
\end{equation}%
i.e. Mehler's formula, to derive the equation%
\begin{eqnarray}
&&\frac{1}{\sqrt{1-4t^{2}}}\exp \left\{ x^{2}-\frac{\left( x-2zt\right) ^{2}%
}{1-4t^{2}}\right\}  \nonumber \\
&=&\exp \left\{ 4xyt-4y^{2}t^{2}\right\} \sum_{n=0}^{\infty }H_{n}\left(
x-2yt\right) H_{n}\left( z-y\right) \frac{t^{n}}{n!}.  \label{41}
\end{eqnarray}%
When $y=0$, this turns into $\left( \ref{40}\right) $. In fact, $\left( \ref%
{41}\right) $ follows directly from $\left( \ref{40}\right) $ when replacing 
$x$\ by $x-2yt$\ and $z$\ by $z-y$\ in the series on the left hand side.

\subsubsection{Binomial coefficients}

For the next corollary we use the binomial transform,%
\begin{equation}
\sum_{k=0}^{n}\binom{n}{k}\left( -1\right) ^{k}\binom{p+k}{k}=\left(
-1\right) ^{n}\binom{p}{n}  \label{42}
\end{equation}%
which is a version of the Vandermonde convolution formula \cite{H1, Ri}.
Here $p$\ can be any complex number. The generating function for $a_{k}=%
\binom{p+k}{k}$ is%
\begin{equation}
f\left( t\right) =\left( 1-t\right) ^{-p-1}=\sum_{k=0}^{\infty }\binom{p+k}{k%
}t^{k}  \label{43}
\end{equation}%
with%
\begin{equation}
\frac{1}{1-t}f\left( \frac{-t}{1-t}\right) =\left( 1-t\right)
^{p}=\sum_{n=0}^{\infty }\binom{p}{n}\left( -1\right) ^{n}t^{n}.  \label{44}
\end{equation}%
According to $\left( \ref{3}\right) $ we obtain.

\begin{corollary}
For any complex $p$,%
\begin{equation}
\sum_{n=0}^{\infty }\binom{p+n}{n}H_{n}\left( x\right) \frac{t^{n}}{n!}%
=e^{2xt-t^{2}}\sum_{n=0}^{\infty }\binom{p}{n}H_{n}\left( x-t\right) \frac{%
t^{n}}{n!}.  \label{45}
\end{equation}
\end{corollary}

It is interesting that when $p$\ is a positive integer, the right-hand side
is finite, i.e. we have the closed-form evaluation%
\begin{equation}
\sum_{n=0}^{\infty }\binom{p+n}{n}H_{n}\left( x\right) \frac{t^{n}}{n!}%
=e^{2xt-t^{2}}\sum_{n=0}^{p}\binom{p}{n}H_{n}\left( x-t\right) \frac{t^{n}}{%
n!}.  \label{46}
\end{equation}

\subsubsection{Stirling numbers of the second kind}

We have the following equation for Stirling numbers of the secon kind $%
\left( \cite{D2, GKP}\right) $:%
\begin{equation}
\sum_{k=0}^{n}\binom{n}{k}\QATOPD\{ \} {k}{m}=\QATOPD\{ \} {n+1}{m+1}
\label{47}
\end{equation}%
and inverse binomial transformation of $\left( \ref{47}\right) $ is%
\begin{equation}
\sum_{k=0}^{n}\binom{n}{k}\left( -1\right) ^{n-k}\QATOPD\{ \}
{k+1}{m+1}=\QATOPD\{ \} {n}{m}  \label{48}
\end{equation}%
which can equally well be written%
\begin{equation}
\sum_{k=0}^{n}\binom{n}{k}\left( -1\right) ^{k}\QATOPD\{ \}
{k+1}{m+1}=\left( -1\right) ^{n}\QATOPD\{ \} {n}{m}.  \label{49}
\end{equation}

\begin{corollary}
We have%
\begin{equation}
\sum_{n=0}^{\infty }\QATOPD\{ \} {n+1}{m+1}H_{n}\left( x\right) \frac{t^{n}}{%
n!}=e^{2xt-t^{2}}\sum_{n=0}^{\infty }\QATOPD\{ \} {n}{m}H_{n}\left(
x-t\right) \frac{t^{n}}{n!}.  \label{50}
\end{equation}
\end{corollary}

\begin{proof}
By setting $a_{k}=\QATOPD\{ \} {k+1}{m+1}$ in $\left( \ref{3}\right) $ and
considering $\left( \ref{49}\right) $ we obtain $\left( \ref{50}\right) .$
\end{proof}

For our next corollary we use the Stirling numbers of the second kind
extended for complex argument. Butzer et al $\left( \cite{BU}\right) $
defined the generalized Stirling numbers (Stirling functions of the second
kind) by%
\begin{equation}
\QATOPD\{ \} {\alpha }{n}=\frac{1}{n!}\sum_{k=0}^{n}\binom{n}{k}\left(
-1\right) ^{n-k}k^{\alpha }  \label{51}
\end{equation}%
for any complex number $\alpha \neq 0$. Obviously, this equation can be
regarded as a binomial transform. In order to match $\left( \ref{6}\right) $
we rewrite $\left( \ref{51}\right) $ in the form (see also $\cite{B1}$).%
\begin{equation}
\left( -1\right) ^{n}n!\QATOPD\{ \} {\alpha }{n}=\sum_{k=0}^{n}\binom{n}{k}%
\left( -1\right) ^{k}k^{\alpha }.  \label{52}
\end{equation}%
This binomial identity with $a_{k}=k^{\alpha }$ leads to the following.

\begin{corollary}
For every complex number $\alpha \neq 0$\ we have%
\begin{equation}
\sum_{k=0}^{\infty }k^{\alpha }H_{k}\left( x\right) \frac{t^{k}}{k!}%
=e^{2xt-t^{2}}\sum_{n=0}^{\infty }\QATOPD\{ \} {\alpha }{n}H_{n}\left(
x-t\right) t^{n}.  \label{53}
\end{equation}
\end{corollary}

In the case when $\alpha =m$ is a positive integer, $\QATOPD\{ \} {\alpha
}{n}=\QATOPD\{ \} {m}{n}$ are the usual Stirling numbers of the second kind $%
\left( \cite{H2}\right) $. These numbers have the property $\QATOPD\{ \}
{m}{n}=0$ when $m<n$. Therefore, we obtain the closed form evaluation%
\begin{equation}
\sum_{k=0}^{\infty }k^{m}H_{k}\left( x\right) \frac{t^{k}}{k!}%
=e^{2xt-t^{2}}\sum_{n=0}^{m}\QATOPD\{ \} {m}{n}H_{n}\left( x-t\right) t^{n}
\label{54}
\end{equation}%
for every positive integer $m$. This formula was obtained independently in 
\cite{H2} and \cite{SH} by different means.

\subsubsection{Exponential numbers}

We have the following equation for exponential numbers $\left( \cite{C, D2}%
\right) $%
\begin{equation}
\phi _{n+1}=\sum_{k=0}^{n}\binom{n}{k}\phi _{k}  \label{55}
\end{equation}%
and inverse binomial transformation of $\left( \ref{55}\right) $ is%
\begin{equation}
\phi _{n}=\sum_{k=0}^{n}\binom{n}{k}\left( -1\right) ^{n-k}\phi _{k+1}
\label{56}
\end{equation}%
which can equally well be written%
\begin{equation}
\left( -1\right) ^{n}\phi _{n}=\sum_{k=0}^{n}\binom{n}{k}\left( -1\right)
^{k}\phi _{k+1}.  \label{57}
\end{equation}%
Then we have the following corollary.

\begin{corollary}
We have%
\begin{equation}
\sum_{n=0}^{\infty }\phi _{n+1}H_{n}\left( x\right) \frac{t^{n}}{n!}%
=e^{2xt-t^{2}}\sum_{n=0}^{\infty }\phi _{n}H_{n}\left( x-t\right) \frac{t^{n}%
}{n!}.  \label{58}
\end{equation}
\end{corollary}

\begin{proof}
By setting $a_{k}=\phi _{k+1}$ in $\left( \ref{3}\right) $ and considering $%
\left( \ref{57}\right) $ we obtain $\left( \ref{58}\right) .$
\end{proof}

\subsubsection{Geometric numbers}

We have the following equation for geometric numbers $\left( \cite{C, D2}%
\right) $%
\begin{equation}
2w_{n}=\sum_{k=0}^{n}\binom{n}{k}w_{k}  \label{59}
\end{equation}%
and inverse binomial transformation of $\left( \ref{59}\right) $ is%
\begin{equation}
w_{n}=\sum_{k=0}^{n}\binom{n}{k}\left( -1\right) ^{n-k}2w_{k}  \label{60}
\end{equation}%
which can equally well be written%
\begin{equation}
\left( -1\right) ^{n}w_{n}=\sum_{k=0}^{n}\binom{n}{k}\left( -1\right)
^{k}2w_{k}.  \label{61}
\end{equation}%
Then we have the following corollary.

\begin{corollary}
We have%
\begin{equation}
\sum_{n=0}^{\infty }2w_{n}H_{n}\left( x\right) \frac{t^{n}}{n!}%
=e^{2xt-t^{2}}\sum_{n=0}^{\infty }w_{n}H_{n}\left( x-t\right) \frac{t^{n}}{n!%
}.  \label{62}
\end{equation}
\end{corollary}

\begin{proof}
Setting $a_{k}=2w_{k}$ in $\left( \ref{3}\right) $ and considering $\left( %
\ref{61}\right) $ gives $\left( \ref{62}\right) .$
\end{proof}

\subsubsection{Fibonacci numbers}

Let us consider the generating function of Fibonacci numbers%
\begin{equation}
F\left( t\right) =\frac{t}{1-t-t^{2}}=\sum_{n=0}^{\infty }F_{n}t^{n}.
\label{63}
\end{equation}%
On the other hand we have,%
\begin{equation}
\frac{1}{1-t}F\left( \frac{t}{1-t}\right) =\frac{t}{1-3t+t^{2}}.  \label{64}
\end{equation}%
But we know that $\left( \cite{D1, TK}\right) $:%
\begin{equation}
\frac{t}{1-3t+t^{2}}=\sum_{n=0}^{\infty }F_{2n}t^{n}.  \label{65}
\end{equation}%
Then from $\left( \ref{65}\right) $ and formula $\left( \ref{7}\right) $ we
get%
\begin{equation}
F_{2n}=\sum_{k=0}^{n}\binom{n}{k}F_{k}.  \label{66}
\end{equation}%
Here using inverse binomial transformation we get%
\begin{equation}
F_{n}=\sum_{k=0}^{n}\binom{n}{k}\left( -1\right) ^{n-k}F_{2k}.  \label{67}
\end{equation}%
Equation $\left( \ref{67}\right) $ may also be put in the form%
\begin{equation}
\left( -1\right) ^{n}F_{n}=\sum_{k=0}^{n}\binom{n}{k}\left( -1\right)
^{k}F_{2k}.  \label{68}
\end{equation}

\begin{corollary}
We have%
\begin{equation}
\sum_{n=0}^{\infty }F_{2n}H_{n}\left( x\right) \frac{t^{n}}{n!}%
=e^{2xt-t^{2}}\sum_{n=0}^{\infty }F_{n}H_{n}\left( x-t\right) \frac{t^{n}}{n!%
}.  \label{69}
\end{equation}
\end{corollary}

\begin{proof}
By setting $a_{k}=F_{2k}$ in $\left( \ref{3}\right) $ and considering $%
\left( \ref{68}\right) $ we obtain $\left( \ref{69}\right) $.
\end{proof}

Let us consider the function%
\begin{equation}
\overline{F}\left( t\right) =\frac{-t}{1+t-t^{2}}=\sum_{n=0}^{\infty }\left(
-1\right) ^{n}F_{n}t^{n}.  \label{70}
\end{equation}%
Then we have%
\begin{equation}
\frac{1}{1-t}\overline{F}\left( \frac{t}{1-t}\right) =\frac{-t}{1-t-t^{2}}%
=-\sum_{n=0}^{\infty }F_{n}t^{n}.  \label{71}
\end{equation}%
Now using $\left( \ref{7}\right) $ we get%
\begin{equation}
-F_{n}=\sum_{k=0}^{n}\binom{n}{k}\left( -1\right) ^{k}F_{k}.  \label{72}
\end{equation}

\begin{corollary}
We have%
\begin{equation}
\sum_{n=0}^{\infty }F_{n}H_{n}\left( x\right) \frac{t^{n}}{n!}%
=e^{2xt-t^{2}}\sum_{n=0}^{\infty }\left( -1\right) ^{n+1}F_{n}H_{n}\left(
x-t\right) \frac{t^{n}}{n!}.  \label{73}
\end{equation}
\end{corollary}

\begin{proof}
By setting $a_{k}=F_{k}$ in $\left( \ref{3}\right) $ and considering $\left( %
\ref{72}\right) $ we obtain $\left( \ref{73}\right) .$
\end{proof}

\begin{corollary}
We have%
\begin{equation}
\sum_{n=0}^{\infty }\left( -1\right) ^{n}F_{n}H_{n}\left( x\right) \frac{%
t^{n}}{n!}=e^{2xt-t^{2}}\sum_{n=0}^{\infty }\left( -1\right)
^{n}F_{2n}H_{n}\left( x-t\right) \frac{t^{n}}{n!}.  \label{74}
\end{equation}
\end{corollary}

\begin{proof}
By setting $a_{k}=\left( -1\right) ^{k}F_{k}$ in $\left( \ref{3}\right) $
and considering $\left( \ref{66}\right) $ we obtain $\left( \ref{74}\right)
. $
\end{proof}

Similar transformation formulas can be obtain for Lucas numbers.

\subsubsection{Lucas numbers}

Let us consider the generating function of Lucas numbers%
\begin{equation}
L\left( t\right) =\frac{2-t}{1-t-t^{2}}=\sum_{n=0}^{\infty }L_{n}t^{n}
\label{75}
\end{equation}%
Then we have%
\begin{equation}
\frac{1}{1-t}L\left( \frac{-t}{1-t}\right) =\frac{2-t}{1-t-t^{2}}%
=\sum_{n=0}^{\infty }L_{n}t^{n}.  \label{76}
\end{equation}%
$\left( \ref{76}\right) $ combines with $\left( \ref{8}\right) $ to give%
\begin{equation}
L_{n}=\sum_{k=0}^{n}\binom{n}{k}\left( -1\right) ^{k}L_{k}.  \label{77}
\end{equation}%
Now we can state the following corollary.

\begin{corollary}
We have%
\begin{equation}
\sum_{n=0}^{\infty }L_{n}H_{n}\left( x\right) \frac{t^{n}}{n!}%
=e^{2xt-t^{2}}\sum_{n=0}^{\infty }\left( -1\right) ^{n}L_{n}H_{n}\left(
x-t\right) \frac{t^{n}}{n!}.  \label{78}
\end{equation}
\end{corollary}

\begin{proof}
By setting $a_{k}=L_{k}$ in $\left( \ref{3}\right) $ and considering $\left( %
\ref{77}\right) $ we obtain $\left( \ref{78}\right) .$
\end{proof}

Due to giving more applications let us consider the following generating
function%
\begin{equation}
\overline{L}\left( t\right) =\frac{2+t}{1+t-t^{2}}=\sum_{n=0}^{\infty
}\left( -1\right) ^{n}L_{n}t^{n}.  \label{79}
\end{equation}%
Then we have%
\begin{equation}
\frac{1}{1-t}\overline{L}\left( \frac{-t}{1-t}\right) =\frac{2-3t}{1-3t+t^{2}%
}.  \label{80}
\end{equation}%
But we know that $\left( \cite{D1, TK}\right) $:%
\begin{equation}
\frac{2-3t}{1-3t+t^{2}}=\sum_{n=0}^{\infty }L_{2n}t^{n}.  \label{81}
\end{equation}%
From $\left( \ref{81}\right) $ and $\left( \ref{8}\right) $ it follows that%
\begin{equation}
L_{2n}=\sum_{k=0}^{n}\binom{n}{k}L_{k}.  \label{82}
\end{equation}%
Now using inverse binomial transformation we get%
\begin{equation}
L_{n}=\sum_{k=0}^{n}\binom{n}{k}\left( -1\right) ^{n-k}L_{2k}  \label{83}
\end{equation}%
which can equally well be written%
\begin{equation}
\left( -1\right) ^{n}L_{n}=\sum_{k=0}^{n}\binom{n}{k}\left( -1\right)
^{k}L_{2k}.  \label{84}
\end{equation}%
Now we have the following corollaries.

\begin{corollary}
We have%
\begin{equation}
\sum_{n=0}^{\infty }L_{2n}H_{n}\left( x\right) \frac{t^{n}}{n!}%
=e^{2xt-t^{2}}\sum_{n=0}^{\infty }L_{n}H_{n}\left( x-t\right) \frac{t^{n}}{n!%
}.  \label{85}
\end{equation}
\end{corollary}

\begin{proof}
By setting $a_{k}=L_{2k}$ in $\left( \ref{3}\right) $ and considering $%
\left( \ref{84}\right) $ we obtain $\left( \ref{85}\right) .$
\end{proof}

\begin{corollary}
We have%
\begin{equation}
\sum_{n=0}^{\infty }\left( -1\right) ^{n}L_{n}H_{n}\left( x\right) \frac{%
t^{n}}{n!}=e^{2xt-t^{2}}\sum_{n=0}^{\infty }\left( -1\right)
^{n}L_{2n}H_{n}\left( x-t\right) \frac{t^{n}}{n!}.  \label{86}
\end{equation}
\end{corollary}

\begin{proof}
By setting $a_{k}=\left( -1\right) ^{k}L_{k}$ in $\left( \ref{3}\right) $
and considering $\left( \ref{82}\right) $ we obtain $\left( \ref{86}\right)
. $
\end{proof}

Now we give some results obtained by using transformation formulas $\left( %
\ref{10}\right) $ and $\left( \ref{11}\right) $.

\subsection{Transformation formulas $\left( \protect\ref{10}\right) $\ and $%
\left( \protect\ref{11}\right) $}

Applying the transformation formula $\left( \ref{10}\right) $\ to the
equation $\left( \ref{1}\right) $\ we get the following formula which is a
generalization of $\left( \ref{54}\right) $:%
\begin{equation}
\sum_{n=0}^{\infty }H_{n}\left( x\right) f\left( n\right) \frac{t^{n}}{n!}%
=e^{2xt-t^{2}}\sum_{n=0}^{\infty }\frac{f^{\left( n\right) }\left( 0\right) 
}{n!}\sum_{k=0}^{n}\QATOPD\{ \} {n}{k}t^{k}H_{k}\left( x-t\right) .
\label{87}
\end{equation}%
If we set $f\left( t\right) =t^{m}$ in $\left( \ref{87}\right) $ we get $%
\left( \ref{54}\right) .$

It is possible to obtain more general results than $\left( \ref{54}\right) $%
\ by setting $f\left( t\right) $\ as an arbitrary polynomial of order $m$ as%
\[
f\left( t\right) =p_{m}t^{m}+p_{m-1}t^{m-1}+\cdots +p_{1}t+p_{0} 
\]%
where $p_{0},$ $p_{1},$ $\cdots ,p_{m-1},$ $p_{m}$ are any complex numbers.
Hence we get following equation,%
\begin{eqnarray}
&&\sum_{n=0}^{\infty }\left( p_{m}n^{m}+p_{m-1}n^{m-1}+\cdots
+p_{1}n+p_{0}\right) H_{n}\left( x\right) \frac{t^{n}}{n!}  \label{88} \\
&=&e^{2xt-t^{2}}\sum_{n=0}^{m}p_{n}\sum_{k=0}^{n}\QATOPD\{ \}
{n}{k}t^{k}H_{k}\left( x-t\right) .  \nonumber
\end{eqnarray}

To obtain more general results, let us set $g\left( t\right) =e^{2xt-t^{2}}$
in generalized transformation formula $\left( \ref{11}\right) $. Then we have%
\begin{equation}
\sum_{n=r}^{\infty }H_{n}\left( x\right) \frac{f_{r}\left( n\right) }{\left(
n-r\right) !n^{r}}t^{n}=e^{2xt-t^{2}}\sum_{n=r}^{\infty }\frac{f^{\left(
n\right) }\left( 0\right) }{n!}\sum_{k=0}^{n}\QATOPD\{ \}
{n}{k}_{r}t^{k}H_{k}\left( x-t\right) .  \label{89}
\end{equation}

If we set $f\left( t\right) =t^{m}$ such that $m\geq r$\ in $\left( \ref{89}%
\right) $ we obtain%
\begin{equation}
\sum_{n=r}^{\infty }n^{m-r}H_{n}\left( x\right) \frac{t^{n}}{\left(
n-r\right) !}=e^{2xt-t^{2}}\sum_{k=0}^{m}\QATOPD\{ \}
{m}{k}_{r}t^{k}H_{k}\left( x-t\right)  \label{89+}
\end{equation}%
which is a generalization of $\left( \ref{54}\right) $.

Again to obtain more general formula than $\left( \ref{89+}\right) $\ we set 
$f\left( t\right) =p_{m}t^{m}+p_{m-1}t^{m-1}+\cdots +p_{1}t+p_{0}$ in $%
\left( \ref{89}\right) .$\ Hence we get%
\begin{eqnarray}
&&\sum_{n=r}^{\infty }H_{n}\left( x\right) \frac{\left(
p_{m}n^{m}+p_{m-1}n^{m-1}+\cdots +p_{r}n^{r}\right) }{\left( n-r\right)
!n^{r}}t^{n}  \nonumber \\
&=&e^{2xt-t^{2}}\sum_{n=r}^{m}p_{n}\sum_{k=0}^{n}\QATOPD\{ \}
{n}{k}_{r}t^{k}H_{k}\left( x-t\right) .  \label{90}
\end{eqnarray}

\subsubsection{Results using transformation formula $\left( \protect\ref{10}%
\right) $}

Generating function of Hermite polynomials is an entire function. Therefore
we can consider $f\left( t\right) =e^{2xt-t^{2}}$\ in $\left( \ref{10}%
\right) .$ Then we have%
\begin{equation}
\sum_{n=0}^{\infty }\frac{g^{\left( n\right) }\left( 0\right) }{n!}%
e^{2nx-n^{2}}t^{n}=\sum_{n=0}^{\infty }\frac{H_{n}\left( x\right) }{n!}%
\sum_{k=0}^{n}\QATOPD\{ \} {n}{k}t^{k}g^{\left( k\right) }\left( t\right) .
\label{91}
\end{equation}

$i)$ If we set $g\left( t\right) =t^{m}$ in the formula $\left( \ref{91}%
\right) $ we get%
\begin{equation}
e^{2mx-m^{2}}=\sum_{n=0}^{\infty }\frac{H_{n}\left( x\right) }{n!}%
\sum_{k=0}^{n}\QATOPD\{ \} {n}{k}\left( m\right) _{k}  \label{92}
\end{equation}%
where $\left( m\right) _{k}$ is the Pochhammer symbol, i.e%
\begin{equation}
\left( m\right) _{k}=m\left( m-1\right) \left( m-2\right) \ldots \left(
m-k+1\right) .  \label{Ps}
\end{equation}
Now comparision of the coefficients of the both sides in $\left( \ref{92}%
\right) $ gives the following wellknown equation:%
\begin{equation}
m^{n}=\sum_{k=0}^{n}\QATOPD\{ \} {n}{k}\left( m\right) _{k}  \label{93}
\end{equation}%
Later we generalize $\left( \ref{93}\right) .$

$ii)$ If we set $g\left( t\right) =e^{t}$ in the formula $\left( \ref{91}%
\right) $ we get%
\begin{equation}
\sum_{m=0}^{\infty }e^{2mx-m^{2}}\frac{t^{m}}{m!}=e^{t}\sum_{n=0}^{\infty }%
\frac{H_{n}\left( x\right) \phi _{n}\left( t\right) }{n!}  \label{94}
\end{equation}%
where $\phi _{n}\left( t\right) $ is $n$th exponential polynomial $\left( 
\cite{B3, B4, D2}\right) $. This can equally well be written by means
generating function of Hermite polynomials as:%
\[
\sum_{m=0}^{\infty }\left( \sum_{n=0}^{\infty }\frac{H_{n}\left( x\right) }{%
n!}m^{n}\right) \frac{t^{m}}{m!}=e^{t}\sum_{n=0}^{\infty }\frac{H_{n}\left(
x\right) \phi _{n}\left( t\right) }{n!}. 
\]%
Then we have%
\[
\sum_{n=0}^{\infty }\frac{H_{n}\left( x\right) }{n!}\sum_{m=0}^{\infty }m^{n}%
\frac{t^{m}}{m!}=e^{t}\sum_{n=0}^{\infty }\frac{H_{n}\left( x\right) \phi
_{n}\left( t\right) }{n!}. 
\]%
Hence we get following equation%
\begin{equation}
\left( t\frac{d}{dt}\right) ^{n}e^{t}=e^{t}\phi _{n}\left( t\right) .
\label{95}
\end{equation}%
The equation $\left( \ref{95}\right) $ can be found in \cite{B3}.

$iii)$ If we set $g\left( t\right) =\frac{1}{1-t}$ in the formula $\left( %
\ref{91}\right) $ we get%
\[
\sum_{k=0}^{\infty }e^{2kx-k^{2}}t^{k}=\frac{1}{1-t}\sum_{n=0}^{\infty }%
\frac{H_{n}\left( x\right) w_{n}\left( \frac{t}{1-t}\right) }{n!} 
\]%
where $w_{n}\left( t\right) $ is $n$th geometric polynomial $\left( \cite%
{B3, D2}\right) $. Now this can equally well be written%
\[
\sum_{k=0}^{\infty }\left( \sum_{n=0}^{\infty }\frac{H_{n}\left( x\right) }{%
n!}k^{n}\right) t^{k}=\frac{1}{1-t}\sum_{n=0}^{\infty }\frac{H_{n}\left(
x\right) w_{n}\left( \frac{t}{1-t}\right) }{n!}. 
\]%
By rearranging we get%
\[
\sum_{n=0}^{\infty }\frac{H_{n}\left( x\right) }{n!}\sum_{k=0}^{\infty
}k^{n}t^{k}=\frac{1}{1-t}\sum_{n=0}^{\infty }\frac{H_{n}\left( x\right)
w_{n}\left( \frac{t}{1-t}\right) }{n!}. 
\]%
This can equally well be written by means of $\left( t\frac{d}{dt}\right) $
operator as%
\[
\sum_{n=0}^{\infty }\frac{H_{n}\left( x\right) }{n!}\left( t\frac{d}{dt}%
\right) ^{n}\frac{1}{1-t}=\frac{1}{1-t}\sum_{n=0}^{\infty }\frac{H_{n}\left(
x\right) }{n!}w_{n}\left( \frac{t}{1-t}\right) 
\]%
Then we have%
\begin{equation}
\left( t\frac{d}{dt}\right) ^{n}\frac{1}{1-t}=\frac{1}{1-t}w_{n}\left( \frac{%
t}{1-t}\right) .  \label{96}
\end{equation}%
The equation $\left( \ref{96}\right) $ also can be found in \cite{B3}.

\subsubsection{Results using transformation formula $\left( \protect\ref{11}%
\right) $}

We can generalize previous results by considering generalized transformation
formula $\left( \ref{11}\right) .$

Let us take $f\left( t\right) =e^{2xt-t^{2}}$\ in $\left( \ref{11}\right) .$
Then we have%
\begin{equation}
\sum_{n=r}^{\infty }\frac{g^{\left( n\right) }\left( 0\right) }{n!}\binom{n}{%
r}\frac{r!}{n^{r}}\left( \sum_{s=r}^{\infty }H_{s}\left( x\right) \frac{n^{s}%
}{s!}\right) t^{n}=\sum_{n=r}^{\infty }\frac{H_{n}\left( x\right) }{n!}%
\sum_{k=0}^{n}\QATOPD\{ \} {n}{k}_{r}t^{k}g^{\left( k\right) }\left(
t\right) .  \label{97}
\end{equation}

$i)$ If we set $g\left( t\right) =t^{m}$ ,$\left( m\geq r\right) $ in the
formula $\left( \ref{97}\right) $ we get%
\begin{equation}
\sum_{n=r}^{\infty }\frac{H_{n}\left( x\right) }{n!}\binom{m}{r}%
r!m^{n-r}=\sum_{n=r}^{\infty }\frac{H_{n}\left( x\right) }{n!}%
\sum_{k=0}^{n}\QATOPD\{ \} {n}{k}_{r}\binom{m}{k}k!.  \label{98}
\end{equation}%
From $\left( \ref{98}\right) $\ we get generalization of $\left( \ref{93}%
\right) $ as follows:%
\begin{equation}
\left( m\right) _{r}m^{n-r}=\sum_{k=0}^{n}\QATOPD\{ \} {n}{k}_{r}\left(
m\right) _{k}.  \label{99}
\end{equation}%
In $\left( \ref{99}\right) $\ we use Pochhammer symbol that we explain in $%
\left( \ref{Ps}\right) $.

$ii)$ If we set $g\left( t\right) =e^{t}$ in the formula $\left( \ref{97}%
\right) $ we get%
\begin{equation}
\sum_{s=r}^{\infty }\frac{H_{s}\left( x\right) }{s!}\sum_{n=r}^{\infty }%
\frac{n^{s-r}t^{n}}{\left( n-r\right) !}=e^{t}\sum_{n=r}^{\infty }\frac{%
H_{n}\left( x\right) _{r}\phi _{n}\left( t\right) }{n!},  \label{100}
\end{equation}%
where $_{r}\phi _{n}\left( t\right) $ is $n$th $r$-exponential polynomial $%
\left( \cite{D4}\right) $. After rearranging $\left( \ref{100}\right) $ we
get%
\begin{equation}
\sum_{n=r}^{\infty }\frac{H_{n}\left( x\right) }{n!}\left( t\frac{d}{dt}%
\right) ^{n-r}t^{r}e^{t}=e^{t}\sum_{n=r}^{\infty }\frac{H_{n}\left( x\right)
_{r}\phi _{n}\left( t\right) }{n!}.  \label{101}
\end{equation}%
Equation $\left( \ref{101}\right) $ gives generalization of $\left( \ref{95}%
\right) $ as%
\begin{equation}
\left( t\frac{d}{dt}\right) ^{n-r}t^{r}e^{t}=_{r}\phi _{n}\left( t\right)
e^{t}  \label{102}
\end{equation}

$iii)$ If we set $g\left( t\right) =\frac{1}{1-t}$ in the formula $\left( %
\ref{97}\right) $ we get%
\begin{equation}
\sum_{n=r}^{\infty }\frac{H_{n}\left( x\right) }{n!}\sum_{m=r}^{\infty }%
\binom{m}{r}r!m^{n-r}t^{m}=\sum_{n=r}^{\infty }\frac{H_{n}\left( x\right) }{%
n!}\frac{_{r}w_{n}\left( \frac{t}{1-t}\right) }{1-t}  \label{103}
\end{equation}%
where $_{r}w_{n}\left( t\right) $ is $n$th $r$-geometric polynomial $\left( 
\cite{D4}\right) $. By rearranging $\left( \ref{103}\right) $ we get%
\begin{equation}
\sum_{n=r}^{\infty }\frac{H_{n}\left( x\right) }{n!}r!\left( t\frac{d}{dt}%
\right) ^{n-r}\frac{t^{r}}{\left( 1-t\right) ^{r+1}}=\sum_{n=r}^{\infty }%
\frac{H_{n}\left( x\right) }{n!}\frac{_{r}w_{n}\left( \frac{t}{1-t}\right) }{%
1-t}.  \label{104}
\end{equation}%
From $\left( \ref{104}\right) $ we see that%
\begin{equation}
r!\left( t\frac{d}{dt}\right) ^{n-r}\frac{t^{r}}{\left( 1-t\right) ^{r+1}}=%
\frac{1}{1-t}\text{ }_{r}w_{n}\left( \frac{t}{1-t}\right)  \label{105}
\end{equation}%
which is a generalization of $\left( \ref{96}\right) $.

\end{document}